\documentclass[a4paper]{amsart} 
\usepackage{amssymb,times, amscd,amsmath,amsthm, xypic}
\usepackage{tikz}
\usepackage{tikz-cd}
\usepackage[all]{xy}
\usepackage{geometry}
\usepackage{mathrsfs}
\makeatletter
\global\let\tikz@ensure@dollar@catcode=\relax
\makeatother
\usepackage{color}
\newcommand{\Z}{\mathbb{Z}}

\newcommand{\E}{\mathscr{E}}

\numberwithin{equation}{section}
\newtheorem{thm}[equation]{Theorem}
\newtheorem{prop}[equation]{Proposition}

\newtheorem{cor}[equation]{Corollary}

\newtheorem{lem}[equation]{Lemma}
\theoremstyle{definition}

\newtheorem{defi}[equation]{Definition}

\newtheorem{rem}[equation]{Remark}

\def\op{\operatorname}

\begin{document}
\title[The group of homotopy self-equivalences is a Lax functor]{The group of homotopy self-equivalences\\ is a Lax functor} 
\author{Toshihiro Yamaguchi and Shoji Yokura}

\thanks {
\noindent
\emph{keywords} : homotopy self-equivalence, correspondence, 2-category, Lax functor }

\address{Faculty of Education, Kochi University, 2-5-1,Kochi, 780-8520, Japan} 
\email{tyamag@kochi-u.ac.jp}

\address{Graduate School of Science and Engineering, Kagoshima University, 1-21-35 Korimoto, Kagoshima, 890-0065, Japan
}

\email{yokura@sci.kagoshima-u.ac.jp}

\maketitle

\begin{abstract}
The group $\E(X)$ of homotopy self-equivalences of a topological space $X$ is a well-known group in homotopy theory and has been studied by many people since it was first introduced in the late 1950s. $\E$ is not a functor in the usual sense. In this paper we show that $\E$ is a Lax functor from the category $\mathscr Top$ of topological spaces to a strict $2$-category $\op{Corr}_{\mathscr Gr}$ of \emph{correspondences} of groups.
\end{abstract}
\section{Introduction}
An important topic in topology and geometry is to find a topological invariant of a space, i.e., a quantity or property which is not changed by a homeomorphism. To be more precise, let $\mathscr I$ be such an invariant, which means that if $f: X \cong Y$ is a homeomorphism, then $\mathscr I(X) = \mathscr I(Y)$, where $=$ is the equality if $\mathscr I$ is a property or an isomorphism $\cong$ if $\mathscr I$ is some algebraic object such as a group, a ring or module etc. More coarsely, we look for a homotopical invariant, i.e., if $f: X \sim Y$ is a homotopy equivalence, 
then $\mathscr I(X) = \mathscr I(Y)$. The most important, simple and fundamental invariant is Euler--Poincar\'e characteristic $\chi(X)$, which is an integer. In a sense this invariant is a generalization of  counting (e.g., see \cite{Baez}). This invariant has been used in many areas, even in physics and mathematical physics. 

The Euler--Poincar\'e characteristic is ``categorified" by the homology or cohomology theory (e.g., see \cite{LS}), which is the fundamental and important topological and furthermore homotopical invariant. The homology theory $H_*$ is a \emph{covariant functor} from the category $\mathscr Top$ of topological spaces to the category $\mathscr Ab$ of abelian groups and the cohomology theory $H^*$ is a \emph{contravariant functor} from $\mathscr Top$ to $\mathscr Ab$. Another important functor, considered as a nice partner to the homology and cohomology theory, is the homotopy groups $\pi_i(X)$ for $i =1, 2, 3, \cdots$. This group is a covariant functor and is defined in a way much easier than the definition of homology and cohomology groups, but it is much harder than the (co)homology groups to compute. For example, we do not know a general structure of the general homotopy groups $\pi_i(S^n)$ of sphere $S^n$, although we \emph{do} know the rational homotopy groups $\pi_i(S^n) \otimes \mathbb Q$ very well by the fundamental work of Jean-Pierre Serre \cite{Serre}:
$$
\pi_i(S^n) \otimes \mathbb Q= \begin{cases}  \mathbb Q, \quad i =n, & \text{if $n$ is an odd integer,}\\
 \mathbb Q, \quad i =n, 2n-1 & \text{if $n$ is an even integer}.
\end{cases} 
$$
The above homotopy group $\pi_i(X)$ is the same as $[S^i, X]_*$\footnote{If you switch $S^i$ and $X$, then $\pi^i(X):=[X, S^i]_*$ is called the \emph{cohomotopy group} and is a contravariant functor.} which is the group of homotopy classes of continuous pointed maps, i.e., continuous map preserving base points. For arbitrary spaces $X$ and $Y$ we can also consider the set $[X, Y]$ of homotopy classes of continuous maps from $X$ to $Y$ without requiring preserving base points. This set $[-,-]$ is a \emph{bifunctor} from the category of topological spaces to the category of sets in the sense that if the source space $X$ is fixed, the assignment $[X, -]$ is a \emph{covariant functor} and if the target space $Y$ is fixed, then $[-,Y]$ is a \emph{contravariant functor}. Thus it is called a \underline{bi}functor. The topic which we discuss in this paper is concerning $[X, X]$, which becomes \emph{the monoid under the composition of maps}. The subset of this monoid $[X,X]$ consisting of invertible elements or units is a group, which is not commutative. This group is denoted by $\E(X)$, called \emph{the group (of homotopy classes) of homotopy self-equivalences of $X$}. $\E(X)$ is clearly a topological and homotopical invariant of topological spaces. This group was for the first time introduced by W.D. Barcus and M.G. Barratt \cite[\S 6]{BB} in 1958 (when it was denoted by $Eq(X)$) and has been studied by many people (e.g., see \cite{Ark, AC, AL, B, B2, BS, CMR, CV, Fel, Kahn, Kahn2, Lee, LY, MR, Pi, Rut, Ru} etc.) Note that we can consider $[X,Y]_*$ with requiring preserving base points, in which case $[-,-]_*$ is a bifunctor from the category of pointed topological spaces and continuous pointed maps to the category of sets. In particular, we can consider the unit group of $[X,X]_*$, which is denoted by $\E(X)_*$. For example it is well-known that $\E(S^n)=\E(S^n)_* = \{[\op{id}_{S^n}], -[\op{id}_{S^n}]\}$ for any $n$, where $-[\op{id}_{S^n}]$ is the homotopy class of any reflection across a hyperplane.

Although $\E(X)$ seems to be homotopy theoretically very simple and quite natural, unlike the (co)homology groups and the homotopy groups, $\E$ \emph{is not a functor} from $\mathscr Top$ to the category $\mathscr Gr$ of groups (not necessarily abelian groups), therefore it is not easy to compute or deal with $\E(X)$. If $f: X \to Y$ is a homotopy equivalence, we can consider the pushforward $\E_*(f):\E(X) \to \E(Y)$ by $\E_*(f)([h_X]):=[f][h_X][f]^{-1}$ for $[h_X] \in \E(X)$. Here $[f]^{-1} := [g]$ for $g:Y \to X$ such that $g \circ f \sim \op{id}_X$ and $f \circ g \sim \op{id}_Y$. Similarly we can define the pullback $\E^*(f):\E(Y) \to \E(X)$ by $\E^*(f)([h_Y]):=[f]^{-1}[h_Y][f]$ for $[h_Y] \in \E(Y)$. But for a general map $f:X \to Y$ one cannot come up with such a pushforward or pullback.  Note that for a homotopy equivalence both maps $\E_*(f)$ and $\E^*(f)$ are isomorphisms.

In this paper we show that $\E$ can be captured as a ``functor", not a (covariant) functor in the usual sense, i.e., not $\E(g \circ f)=\E(g) \circ \E(f)$, but as \emph{a Lax functor}, i.e., 
\begin{equation}\label{lax}
\E(g) \copyright \E(f) \Longrightarrow \E(g \circ f)
\end{equation}
A bit more precise, we define the group $\E(f)$ of homotopy classes of homotopy self-equivalences of a map $f:X \to Y$ (see Definition \ref{E(f)}).
\emph{Our contribution or new perspective} to the study of the group $\E(X)$ is 
to \emph{consider $\E(f)$ as what is so called ``a correspondence" of groups}
\begin{equation}\label{corr-group}
 \E(X) \xleftarrow {p_1} \E(f) \xrightarrow {p_2} \E(Y),
 \end{equation}
\emph{where $\E(f)$ is a subgroup of the product $\E(X) \times \E(Y)$ and $p_1: \E(f) \to \E(X)$ and $p_2: \E(f) \to \E(Y)$  are the restriction of the projections $\pi_1:\E(X) \times \E(Y) \to \E(X)$ and $\pi_2:\E(X) \times \E(Y) \to \E(Y)$  to the subgroup $\E(f)$. This correspondence (\ref{corr-group}) is considered as a morphism from $\E(X)$ to $\E(Y)$.}

In general a pair of homomorphisms $G \xleftarrow {p_1} S \xrightarrow {p_2} F$ of group homomorphisms $p_1:S \to F$ and $p_2:S \to G$ of three groups $G, F$ and $S$, where $S$ is not necessarily a subgroup of the product $F \times G$. Such a pair is called \emph{a span or roof of groups}, and usually depicted as follows (like a ``roof"):
\begin{equation*}
\xymatrix@!0{
& S \ar[dl]_{p_1} \ar[dr]^{p_2} &\\
G && F
}
\end{equation*}
The composition of two spans $G \leftarrow S \rightarrow F$ and $F \leftarrow S' \rightarrow K$ is defined by using the fiber product $S' \times_F S$ (note that the order of $S$ and $S'$) as follows:
\begin{equation}\label{f-product}
\xymatrix@!0{
&& S' \times_F S \ar[dl] \ar[dr]  &&& \\
& S \ar[dl] \ar[dr] && S' \ar[dl] \ar[dr]  &\\
G  && F && K
}
\end{equation}
Usually this composition is denoted by $(G \leftarrow S \rightarrow F) \circ (F \leftarrow S' \rightarrow K)$ or simply denoted by $S' \circ S$ or $G \leftarrow S' \circ S \rightarrow K$. However note that this usual compositions of spans or roofs is \emph{not a nice composition} for our approach, because this composition of correspondences of groups is \emph{not necessarily a correspondence of groups}. Namely, even if $G \leftarrow S \rightarrow F$ and $F \leftarrow S' \rightarrow K$ are correspondences, thus $S$ is a subgroup of $G \times F$ and $S'$ is a subgroup of $F \times K$, the result $G \leftarrow S' \circ S \rightarrow K$ is \emph{not a correspondence of groups}, namely $S' \circ S$ is \emph{not a subgroup of $G \times K$} (although it might be the case when it is isomorphic to a subgroup of $G \times K$). To remedy this drawback, we modify the composition by considering ``fiber-duct" (see below), not the usual fiber product. That is why we denote this new composition $S' \copyright S$. This process is depicted as follows, mimicking the above (\ref{f-product}) by taking he fiber product:

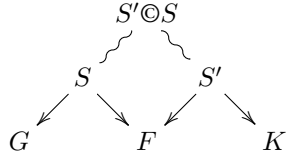
\begin{equation*}
\xymatrix@!0{
&& S' \copyright S \ar@{~}[dl] \ar@{~}[dr]  &&& \\
& S \ar[dl] \ar[dr] && S' \ar[dl] \ar[dr]  &\\
G  && F && K
}
\end{equation*}

Given two groups $G$ and $F$, the set $\op{Corr}_{\mathscr Gr}(G, F)$ of all correspondences $G \leftarrow S \rightarrow F$ with $S$ being a subgroup of $G \times F$ becomes a poset (partially ordered set) by the order $\leq$ defined by 
$$(G \leftarrow S \rightarrow F) \leq (G \leftarrow S' \rightarrow F) \Longleftrightarrow S \subset S'.$$
As well-known, a poset $(P, \leq)$ is also considered as a (small) category, called \emph{poset category}, by setting:
\begin{enumerate}
\item $Obj(P) = P$, i.e., the objects of the category $P$ are all the elements of the poset $P$,
\item For two objects $x, y \in P$, the set of the morphisms from $x$ to $y$ is defined by
$$ Hom_P(x, y)= 
\begin{cases}
\{x \rightarrow y\} \, \text{(only one arrow)}, & \quad \text{if $x \leq y$} \\
\quad   \emptyset, & \quad \text{otherwise}
\end{cases}
$$
\end{enumerate}
Hence $\op{Corr}_{\mathscr Gr}(G, F)$ is a poset category, therefore the category $\op{Corr}_{\mathscr Gr}$ whose objects consist of all groups, i.e., $Obj(\op{Corr}_{\mathscr Gr}) = Obj(\mathscr Gr)$, the set $\op{Corr}_{\mathscr Gr}(G, F)$ of ($1$-)morphisms from  $F$ to $G$ are all the group correspondences $G \leftarrow S \rightarrow F$ and for two morphisms $G \leftarrow S \rightarrow F$ and $G \leftarrow S' \rightarrow F$ (which are in $\op{Corr}_{\mathscr Gr}(F, G)$) we have a 2-morphism $(G \leftarrow S \rightarrow F) \Longrightarrow (G \leftarrow S' \rightarrow F)$ if $S \subset S'$, which is depicted as follows:
$$
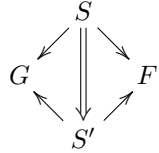
\xymatrix@!0
 {& S \ar[dl] \ar[dr] \ar@{=>}[dd] & \\
  G & & F \\
  & S' \ar[ul] \ar[ur] &
 }
$$
Hence the above (\ref{lax}) means the following that
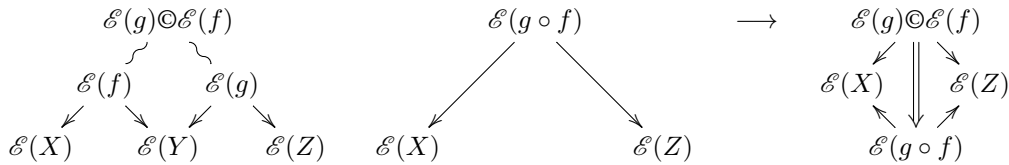
\begin{equation*}
\xymatrix@!0{
&& \E(g) \copyright \E(f) \ar@{~}[dl] \ar@{~}[dr]  &&& \\
& \E(f) \ar[dl] \ar[dr] && \E(g) \ar[dl] \ar[dr]  &\\
\E(X)  && \E(Y) && \E(Z)
}
\xymatrix@!0{
& & \E(g \circ f)  \ar[ddll] \ar[ddrr] & \\
& & & & \\
 \E(X) & & & & \E(Z)
} 
\quad \longrightarrow \quad 
\xymatrix@!0
 {& \E(g) \copyright \E(f)  \ar[dl] \ar[dr] \ar@{=>}[dd] & \\
  \E(X)  & & \E(Z) \\
  & \E(g \circ f) \ar[ul] \ar[ur] &
 }
\end{equation*}

The organization of the paper is as follows. In \S 2 we give a quick recall of the definition of $\E(X)$ and in \S 3 we define $\E(f)$ for a continuous map $f:X \to Y$ and discuss some properties of it. In \S 4 we define the composition of correspondences of groups by what we call ``fiber-duct" and we show that by defining such a composition we obtain \emph{a strict $2$-category of correspondences of groups} (Theorem \ref{st-2-cat}) and we show that \emph{the assignment $\E$ is a Lax functor} (Theorem \ref{lax-fun}). Furthermore we discuss some category-theoretical aspects related to the Lax functor $\E$. In the final section \S 5 we discuss pairs of maps $f:X \to Y$ and $g:Y \to Z$ such that the equality $\E(g) \copyright \E(f) = \E(g \circ f)$ holds. For example, we show that the equality $\E(g) \copyright \E(f) = \E(g \circ f)$ holds for the Hopf fibration $S^1 \overset {f} \hookrightarrow S^3 \xrightarrow g S^2$, but not for the other two Hopf fibrations $S^3 \overset {f} \hookrightarrow S^7 \xrightarrow g S^4$ and $S^7 \overset {f} \hookrightarrow S^{15} \xrightarrow g S^8$.
We also show that the equality holds for the loop-path fibration $\Omega X \overset {f} \hookrightarrow PX \xrightarrow g X$ for a path connected space $X$ and the universal principal $G$-bundle $G \overset {f} \hookrightarrow EG \xrightarrow g BG$, constructed by John W. Milnor \cite{Milnor, Milnor2}, where $BG$ is the classifying space of a topological group $G$. A key for these two results is that $PX$ and $EG$ are both contractible spaces. In fact we show that for two continuous maps $f: X \to Y$ and $g:Y \to Z$ such that $Y$ is a contractible space and $Z$ is path connected, the equality $\E(g) \copyright \E(f) = \E(g \circ f)$ holds. Here the composition $X \xrightarrow f Y \xrightarrow g Z$ is not necessarily a fibration.

\section{The group $\E(X)$ of homotopy self-equivalences of a space $X$}
For two topological spaces $X$ and $Y$ a continuous map $f: X \to Y$ is called a homotopy equivalence if there exists a continuous map $f: Y \to X$ such that $g \circ f \simeq \op{id}_X$ and $f \circ g \simeq \op{id}_Y$. In the case when $X =Y$, $f:X \to X$ is called a \emph{homotopy self-equivalence}\footnote{It seems that in most papers studying $\E(X)$ it is called ``self-homotopy equivalence", but in this paper it is called homotopy self-equivalence, following \cite{Ru}.} and the set of all the homotopy classes of such homotopy self-equivalences of $X$ is denoted by 
$\E(X)$. Namely, it is defined as
\begin{equation}\label{e(x)}
\E(X):= \left \{ [f] \in [X, X] \, | \, \exists g:X \to X; g\circ f \simeq \op{id}_X, f \circ g \simeq \op{id}_X \right \}, 
\end{equation}
which is a group by the composition of maps. 

Here we are a bit sloppy. This group usually depends on whether one consider maps of \emph{pointed} spaces or not. For $\E(X)$ we do not consider a base point. If we consider a base point, the above group is denoted by $\E_*(X)$, the subscript $*$ indicating that we are considering a base point.

Here we take a bit closer look at $\E(X)$. For two topological spaces $X$ and $Y$, the set of homotopy classes of continuous maps from $X$ to $Y$ is denoted by $[X, Y]$.
In the case when $X=Y$, this set $[X, X]$ becomes a non-commutative monoid with the usual composition of map $[f][g]:=[f \circ g]$, namely it satisfies
\begin{enumerate}
\item (associativity) $([f][g])[h] = [f]([g][h])$ for $[f], [g], [h] \in [X, X]$,
\item (the unit) $[\op{id}_X][f]=[\op{id}_X \circ f]=[f]$ and $[f][\op{id}_X]=[f \circ \op{id}_X]=[f].$
\end{enumerate}
However in general $[f][g] \not = [g][f]$, thus it is non-commutative.

$\E(X)$ consists of all the elements $[f] \in [X,X]$ which have the homotopical inverse element $[g]$, i.e., 
$$[f][g]=[g][f]=[\op{id}_X],$$
which is the above requirements $g\circ f \simeq \op{id}_X, f \circ g \simeq \op{id}_X$
defining $\E(X)$ as in (\ref{e(x)}).
Therefore $\E(X)$ is nothing but \emph{the unit group\footnote{Given a monoid $\mathscr M$ the unit group $\mathscr M^{\times}$ of $\mathscr M$ consists of all the elements $x$'s which have their inverses $y$, i.e., $xy=yx=e$, thus it is sometimes called the group of  the invertible elements. For examples, the set $\mathbb Z$ of all the integers is a commutative monoid with respect to the multiplication $\times $ and its unit group $\mathbb Z^{\times} = \{1, -1\}$. The set $\mathbb Z_{>0}$ of all the positive integers is a commutative monoid with respect to the multiplication $\times$ and its unit group $(\mathbb Z_{>0})^{\times} =\{1 \}$. }  $[X,X]^{\times}$} of the above monoid $[X, X]$. It is the largest group in $[X, X]$. 
For example, it is well-known that $\E(S^n) = \E(S^n)_* \cong \mathbb Z_2=\{1, -1\}$ (depending on the mapping degree $1$ or $-1$). Indeed, $[S^n, S^n]= \{ n [\op{id}_{S^n}] \, \, | \, \, n \in \mathbb Z \} \cong \mathbb Z$ where for a negative integer $n$,  $n[\op{id}_{S^n}]= |n|[R_{S^n}]$, here $R_{S^n}:S^n \to S^n$ is a (in fact, any)  reflection map of $S^n$ along a hyperplane and the homotopy class $[R_{S^n}]$ is usually denoted by $-[\op{id}_{S^n}]$.  Then the unit group $\E(S^n)$ of the monoid $[S^n, S^n]$ is $\{1, -1 \} =\{[\op{id}_{S^n}], -[\op{id}_{S^n}]\}$, just like the unit group of \emph{the monoid $\mathbb Z$ (with respect to the multiplication, not the addition)} is $\{1, -1\}$. 

\section{The group $\E(f)$ of homotopy self-equivalences of a map $f$}
$\E(X)$ is defined for a topological space $X$. We define $\E(f)$ for a continuous map $f$ as follows:
\begin{defi}\label{E(f)}(cf. \cite{Hilton}, \cite{Hurvitz}, \cite{Lee}, \cite{LY})
For a continuous map $f: X \to Y$ we set
\begin{equation*}
\E(f) := \left \{([h_X], [h_Y]) \in \E(X) \times \E(Y) \, \, | \, \, f \circ h_X \simeq h_Y \circ f \right  \},
\end{equation*}
namely we consider the following homotopy commutative diagram
\begin{equation}\label{e(f)-dia}
\xymatrix
{ X \ar[r]^{h_X} \ar[d]_f & X \ar[d]^f \\
Y \ar[r]_{h_Y} & Y.
}
\end{equation}
$\E(f)$ shall be called \emph{the group of (homotopy classes) of homotopy self-equivalences of a map $f$}.
\end{defi}

Clearly $\E(f)$ becomes a group by considering compositions component-wise, i.e., 
$$([h_1], [g_1]) \cdot ([h_2], [g_2]) :=([h_1][h_2], [g_1][g_2]).$$

\begin{rem}
A homotopy self-equivalence $f:X \to X$ of a space $X$ can be considered as the following homotopy commutative diagram:
\begin{equation*}\label{e(id)-dia}
\xymatrix
{ X \ar[r]^f \ar[d]_{\op{id}_X} & X \ar[d]^{\op{id}_X}  \\
X \ar[r]_f & X.
}
\end{equation*}
Thus we have
$$\E(\op{id}_X) = \Delta (\E(X)) := \left \{([f], [f]) \in \E(X) \times \E(X) \right \} \cong \E(X).$$
\end{rem}

Here is another approach to $\E(f)$. For two topological spaces $X$ and $Y$, we consider the following action of $\E(X) \times \E(Y)$ on the set $[X, Y]$
$$\mathscr {RL}: \Bigl (\E(X) \times \E(Y) \Bigr ) \times [X, Y] \to [X, Y] \qquad \mathscr {RL}(([h_X], [h_Y]), [f]):= [h_Y]^{-1}[f][h_X],$$
which shall be called \emph{a right-left action}\footnote{In singularity theory of maps of germs, right-left action (or, left-right action, respectively) is a well-known tool in classifying such maps (e.g., see \cite{Nishi}). For example, two map germs $f, g:(\mathbb R^n, 0) \to (\mathbb R^m, 0)$ are called $C^{\infty}$ right-left (or, left-right, respectively) equivalent if there exist $C^{\infty}$-diffeomorphisms $h_{\mathbb R^n}:(\mathbb R^n, 0)  \cong (\mathbb R^n,0)$ and $h_{\mathbb R^m}:(\mathbb R^m,0) \cong (\mathbb R^m,0)$ such that $f = h_{\mathbb R^m}^{-1} \circ g \circ h_{\mathbb R^n}$ (or, $g = h_{\mathbb R^m}  \circ f \circ h_{\mathbb R^n}^{-1}$, respectively). Thus the following diagram commutes:$\xymatrix{(\mathbb R^n,0) \ar[r]^{h_{\mathbb R^n}}_{\cong} \ar[d]_f & (\mathbb R^n,0) \ar[d]^g\\
(\mathbb R^m,0) \ar[r]_{h_{\mathbb R^m}}^{\cong} & (\mathbb R^m,0)}$.}, meaning that $[h_X] \in \E(X)$ acts from \emph{the right-hand side} and $[h_Y] \in \E(Y)$ acts from \emph{the left-hand side as its inverse}. Then it is clear that for a continuous map $f:X \to Y$, $\E(f)$ is nothing but \emph{the stabilizer $(\E(X) \times \E(Y))_{[f]} = \op{Stab}_{\E(X) \times \E(Y)}([f])$} of $\E(X) \times \E(Y)$ which fixes $[f]$, i.e.,
$$\E(f) = \{ ([h_X], [h_Y]) \in \E(X) \times \E(Y) \, \, | \, \, \mathscr {RL}(([h_X], [h_Y]), [f]) = [h_Y]^{-1}[f][h_X]= [f]\}.$$
Here note that $[h_Y]^{-1}[f][h_X]= [f]$ is the same as $[f][h_X]= [h_Y][f]$, i.e., (\ref{e(f)-dia}).
\begin{rem}
Let $\mathscr O(f)$ be the orbit of $[f]$ by the right-left action of the group $\E(X) \times \E(Y)$. Then the well-known Orbit-Stabilizer theorem is the following bijection
$$\Bigl (\E(X) \times \E(Y) \Bigr)/\E(f) \cong \mathscr O(f).$$
\end{rem}
By the definition of $\E(f) (\subset \E(X) \times \E(Y))$, in general it cannot be treated as a functor. If it were a functor, say a covariant functor, we should have an ``associated" homomorphism $\widetilde {\E(f)}:\E(X) \to \E(Y)$. On the other hand, this homomorphism $\widetilde {\E(f)}$ gives rise to a subgroup of $\E(X) \times \E(Y)$ by considering its graph $\widetilde {\E(f)}$; i.e.,
$$\Gamma_{\widetilde {\E(f)}}:=\{ ([h_X], \widetilde {\E(f)}([h_X]) ) \, \, | \, \, [h_X] \in \E(X) \}.$$
In \cite{LY} Jin-Ho Lee and Toshihiro Yamaguchi introduces the notion of $\E$ -map and co-$\E$-map as follows:
\begin{defi} A continuous map $f: X \to Y$ is called an $\E$-map if there is a homomorphism $\phi_f:\E(X) \to \E(Y)$ such that 
\begin{equation}\label{e-map}
\Gamma_{\phi_f}= \{ ([g], \phi_f([g])) \, \, | \, \, [g] \in \E(X) \} \subset \E(f).
\end{equation}
Similarly, a continuous map $f: X \to Y$ is called a co-$\E$-map if there is a homomorphism $\psi_f:\E(Y) \to \E(X)$ such that 
\begin{equation}\label{co-e-map}
 \Gamma_{\psi_f}^t = \{ (\psi_f([h]), [h]) \, \, | \, \, [h] \in \E(Y) \} \subset \E(f)\footnote{Here we note that the graph of the map $\psi_f:\E(Y) \to \E(X)$ is by definition $\{([h_Y], \psi_f([h_Y])) \, \, | \, \, [h_Y] \in \E(Y) \} \subset \E(Y) \times \E(X)$. So by transposing (denoted by superscript $t$) it, we get the above (\ref{co-e-map}).}.
\end{equation}
\end{defi}
It is quite natural to modify the above definition by changing the inclusion $\subset$ to the equality $=$ in (\ref{e-map}) and (\ref{co-e-map}), i.e., as follows:
\begin{defi} A continuous map $f: X \to Y$ is called a \emph{strong} $\E$-map if there is a homomorphism $\phi_f:\E(X) \to \E(Y)$ such that 
\begin{equation*}\label{e-map-1}
\Gamma_{\phi_f}= \{ ([h_X], \phi_f([h_X])) \, \, | \, \, [h_X] \in \E(X) \} = \E(f).
\end{equation*}
Similarly, a continuous map $f: X \to Y$ is called a \emph{strong} co-$\E$-map if there is a homomorphism $\psi_f:\E(Y) \to \E(X)$ such that 
\begin{equation*}\label{co-e-map-2}
\Gamma_{\psi_f}^t = \{ (\psi_f([h_Y]), [h_Y]) \, \, | \, \, [h_Y] \in \E(Y) \} = \E(f).
\end{equation*}
\end{defi}

 If $f:X \to Y$ is a strong $\E$-map or a strong co-$\E$-map, then the group $\E(f)$ has to be a graph of a certain homomorphism from $\E(X) \to \E(Y$ or $\E(Y) \to \E(X)$, respectively. In fact, its converse holds:
\begin{prop}\label{prop-e-maps} A continuous map $f:X \to Y$ is a strong  $\E$-map if and only if the projection map $p_1:\E(f) \to \E(X)$ is an isomorphism as groups. Similarly, a continuous map $f:X \to Y$ is a strong co-$\E$-map if and only if the projection map $p_2:\E(f) \to \E(Y)$ is an isomorphism as groups. 
\end{prop}
\begin{proof} For the sake of simplicity, we just write $g$ instead of the class $[g]$. It suffices to show the case of strong $\E$-map. Let $f:X \to Y$ be a strong  $\E$-map, i.e., we have $\Gamma_{\phi_f}= \{ (g, \phi_f(g)) \, \, | \, \, g \in \E(X) \} = \E(f).$ Then it is clear that $p_1:\E(f)=\Gamma_{\phi_f} \to \E(X)$ is an isomorphism as groups. Conversely we have that $p_1:\E(f) \to \E(X)$ is an isomorphism. For $g \in \E(X)$ we define $\phi_f(g):=p_2 (p_1^{-1}(g))$, where $p_2: \E(f) \to \E(Y)$. This gives us a map $\phi_f:\E(X) \to \E(Y)$. Then $\phi_f$ is a group homomorphism. Indeed, let $p_1^{-1}(g_1)=(g_1, h_1)$ and $p_1^{-1}(g_2)=(g_2, h_2)$. Then we have $(g_1, h_1) \circ (g_2, h_2)=(g_1 \circ g_2, h_1 \circ h_2)$.
Hence $\phi_f(g_1 \circ g_2) = h_1 \circ h_2 = \phi_f(g_1) \circ \phi_f(g_2)$, thus $\phi_f$ is a homomorphism.
\end{proof}
As we observed above, since $\E(\op{id}_X) = \Delta (\E(X)) = \left \{([f], [f]) \in \E(X) \times \E(X) \right \}$, 
we have $\E(\op{id}_X) = \Gamma_{\op{id}_{\E(X)}} = \Gamma^t_{\op{id}_{\E(Y)}}$, hence the identity map $\op{id}_X: X \to X$ is both a strong  $\E$-map and a strong co-$\E$-map. 
\begin{lem}\label{ccc}
If $Y$ is path connected, 
then for any constant map $c:X \to Y$ we have
$$\E(c) = \E(X) \times \E(Y).$$
\end{lem}
\begin{proof} 
Let $c: X \to Y$ be a constant map. Let $[h_X] \in \E(X)$ and $[h_Y] \in \E(Y)$ be arbitrary elements of them respectively and consider the following diagram:
\begin{equation}\label{diagram-c}
\xymatrix{
X \ar[rr]^{h_X} \ar[d]_c && X \ar[d]^c \\
Y \ar[rr]_{h_Y} && Y
}
\end{equation}
Let $c(X) = y_0 \in Y$ and $y_1=h_Y(y_0)$. Since $Y$ is path connected, there is a path $\omega :[0,1] \to Y$ such that $\omega (0)=y_0$ and $\omega(1)=y_1$. Let $\pi_2: X \times [0, 1]$ be the projection to the second factor and let $H:= \omega \circ \pi_2$ be the composition
$$H: X \times [0,1] \xrightarrow {\pi_2} [0,1] \xrightarrow {\omega} Y,$$
 which is clearly a continuous map and $H(x,0)= \left (c \circ h_X \right )(x) = y_0 (\forall x \in X)$ and $H(x, 1)= \left (h_Y \circ c \right )(x)=y_1 (\forall x \in X)$. Therefore $c \circ  h_X$ is homotopic to $h_Y \circ c$, i.e., 
 $$\text{$[c][h_X]=[c \circ h_X] = [h_Y \circ c]=[h_Y][c]$, thus $[h_Y]^{-1}[c][h_X]=[c]$}.$$
Namely the above diagram (\ref{diagram-c}) is a homotopy commutative diagram. Thus $\E(c) = \E(X) \times \E(Y)$. 
\end{proof}
In fact the constant map in the above Lemma \ref{ccc} can be replaced by \emph{any map $f:X \to Y$ homotopic to a constant map}:
\begin{lem}\label{ccc-1} Let $Y$ be path connected. If $f:X \to Y$ is homotopic to a constant map $c:X \to Y$, then we have
$$\E(f) = \E(X) \times \E(Y).$$
\end{lem}
\begin{proof} Let $[h_X] \in \E(X)$ and $[h_Y] \in \E(Y)$ be arbitrary elements. Then we have
\begin{align*}
f \circ h_X & \simeq c \circ h_X \quad \text{(since $f:X \to Y$ is homotopic to $c:X \to Y$)}\\
& \simeq h_Y \circ c \quad \text{(by the proof of Lemma \ref{ccc})}\\
& \simeq h_Y \circ f \quad \text{(again, since $f:X \to Y$ is homotopic to $c:X \to Y$)}.
\end{align*}
Therefore $([h_X], [h_Y]) \in \E(f)$, thus $\E(f) = \E(X) \times \E(Y).$
\end{proof}
\begin{cor}
\begin{enumerate}
\item If $Y$ is path connected and $\E(Y)$ is not a trivial group, then any constant map $c:X \to Y$ cannot be a strong $\E$-map.
\item If $Y$ is path connected and $\E(X)$ is not a trivial group, then any constant map $c:X \to Y$ cannot be a strong co-$\E$-map.
\end{enumerate}
\end{cor}
\begin{proof}
Firstly, since $\E(Y)$ is not trivial, $\E(c)=\E(X) \times \E(Y)$ cannot be the graph of any homomorphism from $\E(X)$ to $\E(Y)$. Thus, any map $c:X \to Y$ \emph{cannot} be a strong $\E$-map.
Secondly, since $\E(X)$ is not a trivial group, there cannot exist a homomorphism $\phi_c:\E(Y) \to \E(X)$ such that $\E(X) \times \E(Y) = \E(c) = \{(\psi_c ([h_Y]), [h_Y]) \, \, | \, [h_Y] \in \E(Y) \}$. Therefore any constant map $f:X \to Y$ \emph{cannot} be a strong co-$\E$-map.
\end{proof}
\begin{rem} In the case of $\E$-map and co $\E$-map, the statements similar to the above Proposition \ref{prop-e-maps} is the following (see \cite{LY}):A continuous map $f:X \to Y$ is an $\E$-map if and only if the projection map $p_1:\E(f) \to \E(X)$ has a section $s:\E(X) \to \E(f)$ as groups. Similarly, a continuous map $f:X \to Y$ is a co-$\E$-map if and only if the projection map $p_2:\E(f) \to \E(Y)$ has a section $s:\E(Y) \to \E(f)$ as groups. 
\end{rem}
\section{$\E$ is a Lax functor}
Even if we define such $\E$-maps and strong $\E$-maps or co $\E$-maps and strong co-$\E$-maps, the collection of such maps do not in general make a category and also even if they make a category, the assignment $\E$ cannot be in general captured as a functor. However, if we consider correspondences as morphisms, then the assignment $\E$ can be captured as \emph{a Lax functor} (e.g., see \cite[Chapter 4]{JY})  as follows.

First we recall the notion of a correspondence of groups.

\begin{defi} Let $G$ and $F$ be two groups and $S$ be a subgroup of the product $G \times F$. Then the following diagram
\begin{equation}\label{corr}
G \xleftarrow {p_1} S \xrightarrow {p_2} F
\end{equation}
is called a \emph{correspondence} of groups. Here $p_1:S \to G$ and $p_2:S \to F$ are the projections to each factor. 
\end{defi}
The correspondence (\ref{corr}) is interchangeably expressed as $S: G \rightsquigarrow F$ (or $S:G \vdash F$).
\begin{rem} 
\begin{enumerate}
\item Given a homomorphism $f:G \to F$, then its graph $\Gamma_f:=\{(g, f(g)) \in G \times F \}$ gives us a correspondence $G \leftarrow \Gamma_f \rightarrow F$, which is a \emph{graph correspondence}.
\item If $S$ is not necessarily a subgroup of $G \times F$, but an arbitrary group and $g:S \to G$ and $f:S \to F$ are group homomorphisms, then the diagram $G \xleftarrow {g} S \xrightarrow {f} F$
is called a \emph{span} of groups.
\end{enumerate}
\end{rem}
\begin{lem}\label{compo} (composition of correspondences) 
Let $G \xleftarrow {p_1} S \xrightarrow {p_2} F$ and $F \xleftarrow {p_1} S' \xrightarrow {p_2} K$ be two correspondences; $S: G \rightsquigarrow F$ and $S':F \rightsquigarrow K$. Then we define the following composition 
\begin{equation}\label{composit}
(F \xleftarrow {p_1} S' \xrightarrow {p_2} K) \copyright (G \xleftarrow {p_1} S \xrightarrow {p_2} F):= G \xleftarrow {p_1} S' \copyright S \xrightarrow {p_2} K, \qquad S' \copyright  S: G \rightsquigarrow K,
\end{equation}
where
\begin{equation*}
S \copyright S':=\{(g, k) \in G \times K \, \, | \, \, \exists f \in F \, \, \text{such that $(g,f) \in S$ and $(f,k) \in S'$} \}.
\end{equation*}
$S \copyright S'$ is a subgroup of $G \times K$, thus (\ref{composit}) is well-defined.

(Note the order of $S$ and $S'$ in $S' \copyright S$, since a correspondence is considered as a kind of ``map".)
\end{lem}
\begin{proof} 
\begin{enumerate}
\item Let $e_G$, $e_F$ and $e_G$ be the identities of $G$, $F$ and $K$, respectively. Since $S$ and $S'$ are subgroups of $G \times F$ and $F \times K$ respectively, we do have that $(e_G, e_F) \in S$ and $(e_F, e_K) \in S'$. Hence we have $(e_G, e_K) \in S' \copyright S$ by the definition of $S' \copyright S$. Hence $S' \copyright S \not = \emptyset.$
\item Let $(g_1, k_1), (g_2, k_2) \in S' \copyright S$. Hence, by the definition of $S' \copyright S$, for each $i =1, 2$ there exist $f_i \in F$ such that $(g_i, f_i) \in S$ and $(f_i, k_i) \in S'$. Since $S$ and $S'$ are both subgroups, we have 
$$\text{$(g_1g_2, f_1f_2) \in S$ and $(f_1f_2, k_1k_2) \in S'$},$$
which implies that $(g_1g_2, k_1k_2) \in S' \copyright S$ by the definition.
\item Let $(g,k) \in S' \copyright S$. Thus it follows from the definition of $S' \copyright S$ there exists $f \in F$ such that $(g,f) \in S$ and $(f,k) \in S'$. Since $S$ and $S'$ are groups, there exist their inverse elements, i.e., $(g,f)^{-1}=(g^{-1}, f^{-1}) \in S$ and $(f,k)^{-1}=(f^{-1}, k^{-1}) \in S'$, respectively.Therefore $(g^{-1}, k^{-1}) \in S' \copyright S$ by the definition. Namely the inverse $(g,k)^{-1}=(g^{-1}, k^{-1}) \in S' \copyright S$.
\end{enumerate}
\end{proof}
\begin{rem} Usually the composition of correspondences (or in general spans or roofs) is defined by the fiber product as follows:
\begin{equation*}
\xymatrix@!0{
&& \hspace{3cm} S \times _F S' \, (\subset G \times F \times F \times K) \ar[dl] \ar[dr]  &&& \\
& S \ar[dl] \ar[dr] && S' \ar[dl] \ar[dr] &\\
G  && F && K
}
\end{equation*}
$$S \times_F S':=\{((g,f), (f, k)) \in S \times S' \, \, | \, \, \exists f \in F \, \, \text{such that $(g,f) \in S$ and $(f,k) \in S'$} \}.$$
Hence, if we consider the projection map $p_1 \times p_4: G \times F \times F \times K \to G \times K$, then our $S' \copyright S$ is nothing but the image $(p_1 \times p_4)(S \times_F S') \subset G \times K.$ In this sense our $S' \copyright S$ is in a sense induced via ``fiber product". So, abusing words, our $S' \copyright S$ is defined by ``fiber-duct". Mimicking the above fiber product diagram, we depict it as follows:
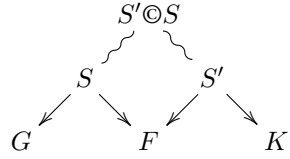
\begin{equation*}
\xymatrix@!0{
&& S' \copyright S \ar@{~}[dl] \ar@{~}[dr]  &&& \\
& S \ar[dl] \ar[dr] && S' \ar[dl] \ar[dr]  &\\
G  && F && K
}
\end{equation*}
The main reason why we use this ``fiber-duct" procedure is that the usual composition (by fiber product) of correspondences \emph{does not} gives rise to a correspondence, in other words $S \times_F S'$ is not necessarily isomorphic to a subgroup of $G \times K$, i.e., the resulting object is what is so called \emph{a span}.  For example let us consider the following simple cases: $G \leftarrow G \times F \rightarrow$ and $F \leftarrow F \times K \rightarrow K$, i.e., we have the following fiber product:
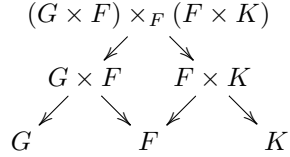
\begin{equation*}
\xymatrix@!0{
&& (G \times F) \times _F (F \times K) \ar[dl] \ar[dr]  &&& \\
& G \times F \ar[dl] \ar[dr] && F \times K \ar[dl] \ar[dr] &\\
G  && F && K
}
\end{equation*}
Then $(G \times F) \times _F (F \times K)= G \times \Delta(F) \times K \cong G \times F \times K$, which is not a subgroup of $G \times K$. However, ``fiber-duct", i.e., deleting the middle components $\Delta(F) \cong F$ in the above fiber product $G \times \Delta(F) \times K \cong G \times F \times K$, gives us the following:
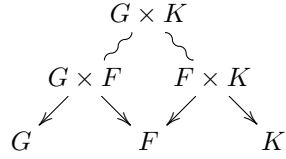
\begin{equation*}
\xymatrix@!0{
&& G \times K \ar@{~}[dl] \ar@{~}[dr]  &&& \\
& G \times F \ar[dl] \ar[dr] && F \times K\ar[dl] \ar[dr]  &\\
G  && F && K.
}
\end{equation*}
\end{rem}

\begin{thm}\label{st-2-cat} (A strict $2$-category of correspondences of groups)  Let $\mathscr Gr$ be the category of groups and let $\op{Corr}_{\mathscr Gr}$ be the category of correspondences of groups. Then $\op{Corr}_{\mathscr Gr}$ becomes a strict $2$-category, to be more precise an enriched category over a poset-category:
\begin{itemize}
\item ($0$-object) $\op{Obj}(\op{Corr}_{\mathscr Gr}) = \op{Obj}(\mathscr Gr)$, i.e., the objects are all the groups,
\item ($1$-morphism) For two groups $G, F$, $\op{Corr}_{\mathscr Gr}(G, F)$ consists of correspondences from $G$ to $F$,
\item ($2$-morphism) For two correspondences $G \xleftarrow {p_1} S_1 \xrightarrow {p_2} F$, $G \xleftarrow {p_1} S_2 \xrightarrow {p_2} F$, i.e.,  $S_1:G \rightsquigarrow F, S_2:G \rightsquigarrow F$, a $2$-morphism $\alpha: S_1 \Rightarrow S_2$ is defined by $S_1 \subset S_2$:
$$
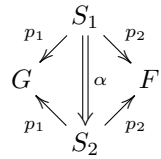
\xymatrix@!0
 {& S_1 \ar[dl]_{p_1} \ar[dr]^{p_2} \ar@{=>}[dd]^\alpha & \\
  G & & F \\
  & S_2 \ar[ul]^{p_1} \ar[ur]_{p_2} &
 }
$$
\end{itemize}
\end{thm}
\begin{proof} It is clear that $\op{Corr}_{\mathscr Gr}(G, F)$ is a small category, because it consists of $S: G \rightsquigarrow F$ for all subgroups $S$ of $F \times G$, thus it is a set. We show that these small categories $\op{Corr}_{\mathscr Gr}(G, F)$'s satisfy the following (for the notion of 2-category, e.g., see \cite{Tate}):
\begin{enumerate}
\item For any group $G$ and the identity map $\op{id}_G:G \to G$, its graph $\Gamma_{\op{id}_G}=$ is denoted by $1_G:=\Delta (G) :=\{(g, g) \, \, | \, g \in G\}$ gives us the graph correspondence $G \leftarrow 1_G \rightarrow G$. So, this is the identity $1$-morphism.
\item For any groups $G,F,K$ we have the above composition given in Definition \ref{compo}
$$\copyright : \op{Corr}_{\mathscr Gr}(F, K) \times \op{Corr}_{\mathscr Gr}(G, F) \to \op{Corr}_{\mathscr Gr}(G, K)$$  
and we have the following commutative diagrams:
\begin{enumerate}
\item for $E, G, F, K$, (for the sake of space, $\op{Corr}_{\mathscr Gr}$ is simply denoted by $\op{Corr}$.)
$$
\xymatrix
{ \op{Corr}(F, K) \times \op{Corr}(G, F)  \times \op{Corr}(E, G)  \ar[d]_{\op{id}_{\op{Corr}(F, K)} \times - \copyright -} \ar[rrr]^{\qquad - \copyright- \times \op{id}_{\op{Corr}(E, G) }} &&& \op{Corr}(G, K) \times \op{Corr}(E, G)  \ar[d]^{ - \copyright - } \\
 \op{Corr}(F, K) \times \op{Corr}(E, F)  \ar[rrr]_{ - \copyright - } &&& \op{Corr}(E, K)  
}
$$
The composition $(- \copyright -) \circ (- \copyright - \times \op{id}_{\op{Corr}(E, G) })$ is due to the following composition of fiber products:
$$
\xymatrix@!0{
&&& S_3 \copyright (S_2 \copyright S_1) \ar@{~}[dl] \ar@{~}[ddrr] \\
&& S_2 \copyright S_1 \ar@{~}[dl] \ar@{~}[dr]  &&& \\
& S_1 \ar[dl] \ar[dr] && S_2 \ar[dl] \ar[dr] && S_3 \ar[dl] \ar[dr] &  \\
E && G  && F && K
}
$$
The composition $(- \copyright -) \circ (\op{id}_{\op{Corr}(F, K)} \times - \copyright-)$ is due to the following composition of fiber products:
$$
\xymatrix@!0{
&&&  (S_3 \copyright S_2) \copyright S_1  \ar@{~}[ddll] \ar@{~}[dr]  \\
&&&& S_3 \copyright S_2 \ar@{~}[dl] \ar@{~}[dr] \\
& S_1 \ar[dl] \ar[dr] && S_2 \ar[dl] \ar[dr] && S_3 \ar[dl] \ar[dr] &  \\
E && G  && F && K
}
$$
Then we have 
$$ S_3 \copyright (S_2 \copyright S_1) =   (S_3 \copyright S_2) \copyright S_1,$$
which means the \emph{associativity of the composition of $1$-morphisms}. Because of this equality this $2$-category is called a \emph{strict $2$-category}\footnote{The composition of $1$-morphisms is \emph{only associative up to a (coherent) $2$-morphism (called the \emph{associator})}, namely  $S_3 \copyright (S_2 \copyright S_1) \cong (S_3 \copyright S_2) \copyright S_1$, then it is called just \emph{$2$-category} or \emph{bicategory}.}.

Indeed we can see this as follows:
\begin{align*}
(e, k)  \in S_3 \copyright (S_2 \copyright S_1) 
& \Longleftrightarrow \text{$\exists f \in F$ such that $(e,f) \in S_2 \copyright S_1$ and $(f,k) \in S_3$} \\
& \Longleftrightarrow \text{$\exists g \in G$, $\exists f \in F$ such that $(e,g) \in S_1$, $(g,f) \in S_2$ and $(f,k) \in S_3$} \\
& \Longleftrightarrow \text{$\exists g \in G$ such that $(e,g) \in S_1$, $(g,k) \in S_3 \copyright S_2$} \\
& \Longleftrightarrow (e, k) \in (S_3 \copyright S_2) \copyright S_1.
\end{align*}
\item We have the following commutative diagram:
$$
\xymatrix
{\op{Corr}_{\mathscr Gr}(G, F) \ar@/^30pt/[rr]^{- \copyright 1_G} \ar@/_30pt/[rr]_{1_F \copyright -} \ar[rr]^{\op{id}_{\op{Corr}_{\mathscr Gr}(G, F)}} && \op{Corr}_{\mathscr Gr}(G, F)
}
$$
Namely, for $G \leftarrow S \rightarrow F \in \op{Corr}_{\mathscr Gr}(G, F)$, $G \leftarrow 1_G \rightarrow G$ and $F \leftarrow 1_F \rightarrow F$, we have 
$$S \copyright 1_G = 1_F \copyright S = S.$$
Indeed, $(g,f) \in S$ can be trivially expressed as follows: $\exists g \in G$ such that $(g,g) \in 1_G=\Delta(G)$ and $(g,f) \in S$, hence we have $S \copyright 1_G=S$. Similarly, $(g,f) \in S$ can be trivially expressed as follows: $\exists f \in F$ such that $(g,f) \in S$ and $(f,f) \in 1_F=\Delta(F)$, hence we have $1_F \copyright S = S$. (In other words, the identity $1$-morphism is also the identity of $1$-morphisms.)
\end{enumerate}
\item (identity $2$-morphism) For a $1$-morphism, i.e, a correspondence $G \leftarrow S \rightarrow F$, the identity $2$-morphism $\op{id}_S$ is given by $\op{id}_S$:
$$
\xymatrix@!0
 {& S \ar[dl]  \ar[dr] \ar@{=>}[dd]^{\op{id}_S} &\\
  G & & F \\
  & S \ar[ul] \ar[ur] &
 }
$$
\item (vertical composition of $2$-morphism) For two groups $G, F$, the category $\op{Corr}_{\mathscr Gr}(G,F)$ is a poset-category, i.e., a poset by the order $(G \leftarrow S_1 \rightarrow F) \leq (G \leftarrow S_2 \rightarrow F)$ (or simply denoted by $S_1 \leq S_2$) defined by $S_1 \subset S_2$, and categorically, this is denoted by the double arrow $\Longrightarrow$ (2-morphism) as follows:
$$
\xymatrix@!0
 {& S_1 \ar[dl]  \ar[dr] \ar@{=>}[dd]^\alpha &\\
  G & & F \\
  & S_2 \ar[ul] \ar[ur] &
 }
$$
Therefore 
$$
\op{Corr}_{\mathscr Gr}(G,F)(S_1, S_2) = \begin{cases} 
\{ S_1 \Longrightarrow S_2\} & \quad \text{if $S_1 \subset S_2$}\\
 \emptyset & \quad \text{if $S_1 \not \subset S_2$}\\
\end{cases}
$$
For 2-morphisms $S_1 \overset{\alpha}{\Longrightarrow} S_2$ and $S_2 \overset{\beta}{\Longrightarrow} S_3$, i.e., $S_1 \subset S_2$ and $S_2 \subset S_3$, we have the composition $S_1 \overset{\beta \circ \alpha}{\Longrightarrow} S_3$, i.e., $S_1 \subset S_3$. 
This is depicted as follows:
$$
\xymatrix@!0
 {& S_1 \ar[dl]  \ar[dr] \ar@{=>}[d]^{\alpha} &\\
  G & S_2 \ar@{=>}[d]^{\beta} \ar[l] \ar[r] &   F  \\
  & S_3 \ar[ul] \ar[ur] &
  } \qquad \mapsto \qquad 
\xymatrix@!0
 {& S_1 \ar[dl]  \ar[dr] \ar@{=>}[dd]^{\beta \circ \alpha} &\\
  G & & F \\
  & S_2 \ar[ul] \ar[ur] &
 }
$$
The associativity of the vertical compositions of $2$-morphisms is clear. It is because $S_1 \overset{\alpha}{\Longrightarrow} S_2 \overset{\beta}{\Longrightarrow} S_3 \overset{\gamma}{\Longrightarrow} S_4$, i.e., $S_1 \subset S_2 \subset S_3 \subset S_4$, clearly implies that $(\gamma \circ \beta) \circ \alpha = \gamma \circ (\beta \circ \alpha)$\footnote{$(\gamma \circ \beta) \circ \alpha$ means that $S_1 \subset S_2$ and $S_2 \subset S_4$ implies $S_1 \subset S_4$, on the other hand $\gamma \circ (\beta \circ \alpha)$ means that 
$S_1 \subset S_3$ and $S_3 \subset S_4$ implies $S_1 \subset S_4$. Thus they are the same.}. 
\item (horizontal composition of $2$-morphisms)\label{hor-comp}
$$
\xymatrix@!0
 {& S_1 \ar[dl]  \ar[dr] \ar@{=>}[dd]^\alpha && S'_1 \ar[dl]  \ar[dr] \ar@{=>}[dd]^\beta \\
  G & & F & & K  \\
  & S_2 \ar[ul] \ar[ur] && S'_2 \ar[ul] \ar[ur]
 }
 \qquad \mapsto \qquad 
 \xymatrix@!0
 {&  S'_1 \copyright S_1 \ar[dl]  \ar[dr] \ar@{=>}[dd]^{\beta \copyright \alpha} & \\
  G & & K \\
  & S'_2 \copyright S_2  \ar[ul] \ar[ur] &
 }
$$
This means that $S_1 \subset S_2$ and $S'_1 \subset S'_2$ imply $S'_1 \copyright S_1 \subset S'_2 \copyright S_2$.
This can be seen as follows:
\begin{align*}
(g,k) \in S'_1 & \copyright S_1 \Longleftrightarrow \text{$\exists f \in F$ such that $(g,f) \in S_1$ and $(f,k) \in S'_1$}\\
& \qquad \Longrightarrow \text{$\exists f \in F$ such that $(g,f) \in S_2$ and $(f,k) \in S'_2$}\\
&  \hspace{3cm} \text{(since $S_1 \subset S_2$ and $S'_1 \subset S'_2$)}\\
& \qquad \Longrightarrow (g,k) \in S'_2 \copyright S_2.
\end{align*}
Therefore we have that $S'_1 \copyright S_1 \subset S'_2 \copyright S_2$.
\item (the horizontal composition of the identity $2$-morphisms)
$$
\xymatrix@!0
 {& S \ar[dl]  \ar[dr] \ar@{=>}[dd]^{\op{id}_S} && S' \ar[dl]  \ar[dr] \ar@{=>}[dd]^{\op{id}_{S'}} \\
  G & & F & & K  \\
  & S \ar[ul] \ar[ur] && S' \ar[ul] \ar[ur]
 }
 \qquad \mapsto \qquad 
 \xymatrix@!0
 {&  S' \copyright S \ar[dl]  \ar[drr] \ar@{=>}[dd]^{\op{id}_{S'} \copyright \op{id}_S} & \\
  G & & & K \\
  & S' \copyright S  \ar[ul] \ar[urr] &
 }
$$
It follows from the proof of the above (\ref{hor-comp}) that $\op{id}_{S'} \copyright \op{id}_S= \op{id}_{S' \copyright S}$.
\item (the associativity of the horizontal composition of $2$-morphisms):
$$
\xymatrix@!0
 {& S_1 \ar[dl]  \ar[dr] \ar@{=>}[dd]^\alpha && S'_1 \ar[dl]  \ar[dr] \ar@{=>}[dd]^\beta && S''_1 \ar[dl]  \ar[dr] \ar@{=>}[dd]^\gamma \\
  G & & F & & K & & M \\
  & S_2 \ar[ul] \ar[ur] && S'_2 \ar[ul] \ar[ur] && S''_2 \ar[ul] \ar[ur]
 }
\quad  \mapsto \quad 
 \xymatrix@!0
 {&  S''_1 \copyright (S'_1 \copyright S_1) \ar[dl]  \ar[drr] \ar@{=>}[dd]^{\gamma \copyright (\beta \copyright \alpha)} & \\
  G & & & M \\
  & S''_2 \copyright (S'_2 \copyright S_2)  \ar[ul] \ar[urr] &
 }
 \quad 
 \xymatrix@!0
 {&  (S''_1 \copyright S'_1) \copyright S_1 \ar[dl]  \ar[drr] \ar@{=>}[dd]^{(\gamma \copyright \beta) \copyright \alpha} & \\
  G & & & M \\
  & (S''_2 \copyright S'_2) \copyright S_2  \ar[ul] \ar[urr] &
 }
$$
Since $S''_1 \copyright (S'_1 \copyright S_1)= (S''_1 \copyright S'_1) \copyright S_1$ and $S''_2 \copyright (S'_2\copyright S_2)= (S''_2 \copyright S'_2) \copyright S_2$ as shown above and the $2$-morphisms are the inclusions, thus we see that $\gamma \copyright (\beta \copyright \alpha) = (\gamma \copyright \beta) \copyright \alpha.$
\item (the interchange of horizontal and vertical compositions of $2$-morphisms)
$$
\xymatrix@!0
 {& S_1 \ar[dl]  \ar[dr] \ar@{=>}[d]^{\alpha} &&  S'_1 \ar[dl]  \ar[dr] \ar@{=>}[d]^{\alpha'} \\
  G & S_2 \ar@{=>}[d]^{\beta} \ar[l] \ar[r] & F  & S'_2 \ar@{=>}[d]^{\beta'} \ar[l] \ar[r] & K\\
  & S_3 \ar[ul] \ar[ur] &&  S'_3 \ar[ul] \ar[ur] 
  } \mapsto 
\xymatrix@!0
 {
 & S_1 \ar[dl]  \ar[dr] \ar@{=>}[dd]^{\beta \circ \alpha} && S'_1 \ar[dl]  \ar[drr] \ar@{=>}[dd]^{\beta' \circ \alpha'}\\
  G & & F &&& K\\
  & S_3 \ar[ul] \ar[ur] && S'_3 \ar[ul] \ar[urr] 
 }
 \mapsto 
 \xymatrix@!0
 {
 & S'_1 \copyright S_1 \ar[dl]  \ar[drrr] \ar@{=>}[dd]^{(\beta' \circ \alpha') \copyright (\beta \circ \alpha)} &&& \\
  G & & & &  K\\
  & S'_3 \copyright S_3 \ar[ul] \ar[urrr] 
 }
$$
$$
\xymatrix@!0
 {& S_1 \ar[dl]  \ar[dr] \ar@{=>}[d]^{\alpha} &&  S'_1 \ar[dl]  \ar[dr] \ar@{=>}[d]^{\alpha'} \\
  G & S_2 \ar@{=>}[d]^{\beta} \ar[l] \ar[r] & F  & S'_2 \ar@{=>}[d]^{\beta'} \ar[l] \ar[r] & K\\
  & S_3 \ar[ul] \ar[ur] &&  S'_3 \ar[ul] \ar[ur] 
  }  \mapsto \quad 
\xymatrix@!0
 { & & S'_1 \copyright S_1 \ar[dll]  \ar[drr] \ar@{=>}[d]^{\alpha'\copyright \alpha} &\\
  G && S'_2 \copyright S_2 \ar@{=>}[d]^{\beta' \copyright \beta } \ar[ll] \ar[rr] &&   K  \\
  && S'_3 \copyright S_3 \ar[ull] \ar[urr] &
 }
 \quad \mapsto 
 \xymatrix@!0
 {
 & S'_1 \copyright S_1 \ar[dl]  \ar[drrr] \ar@{=>}[dd]^{(\beta' \circ \beta) \copyright (\alpha' \circ \alpha)} &&& \\
  G & & & &  K\\
  & S'_3 \copyright S_3 \ar[ul] \ar[urrr] 
 }
$$
Then by the same argument above we have that $(\beta' \circ \alpha') \copyright (\beta \circ \alpha) = (\beta' \circ \beta) \copyright (\alpha' \circ \alpha)$.
\end{enumerate}
\end{proof}

\begin{thm}\label{lax-fun} The assignment $\E$ is a Lax functor from the category $\mathscr Top$ of topological spaces and continuous maps to the strict $2$-category $\op{Corr}_{\mathscr Gr}$ of correspondences of groups. Namely, 
\begin{enumerate}
\item for a topological space $X$, we define $\E(X)$, 
\item for a continuous map $f:X \to Y$, we have a correspondence of groups
$$\text{$\E(X) \xleftarrow {p_1}  \E(f) \xrightarrow {p_2} \E(Y)$ \quad  or  \quad $\E(f): \E(X) \rightsquigarrow \E(Y)$},$$
\item for continupus maps $f:X \to Y$ and $g:Y \to Z$, we have a 2-morphism 
$$\text{$\E(g) \copyright \E(f) \Longrightarrow \E(g \circ f)$ \quad or \quad 
 $\E(g) \copyright \E(f) \subset \E(g \circ f)$}.$$
\end{enumerate}
\end{thm}
\begin{proof} For a continuous map $f:X \to Y$ and $g: Y \to Z$, we have the following correspondence of groups
$$\E(X) \xleftarrow {p_1} \E(f) \xrightarrow {p_2} \E(Y), \qquad \E(Y) \xleftarrow {p_1} \E(g) \xrightarrow {p_2} \E(Z).$$
It suffices to show $\E(g) \copyright \E(f) \Longrightarrow \E(g \circ f)$. We have
$$\E(f) =\{ ([h_X], [h_Y] ) \in \E(X) \times \E(Y) \, \, | \, \, [f][h_X] = [h_Y][f] \},$$
namely we have the following homotopy commutative diagram
\begin{equation*}\label{e(f)-gh}
\xymatrix
{ X \ar[r]^{h_X} \ar[d]_f & X \ar[d]^f \\
Y \ar[r]_{h_Y} & Y.
}
\end{equation*}
Similarly we have 
$$\E(g) =\{ ([h_Y], [h_Z] ) \in \E(Y) \times \E(Z) \, \, | \, \, [g][h_Y] = [h_Z][g] \},$$
namely we have the following homotopy commutative diagram
\begin{equation*}\label{e(f)-gh}
\xymatrix
{ Y \ar[r]^{h_Y} \ar[d]_g & Y \ar[d]^g \\
Z \ar[r]_{h_Z} & Z.
}
\end{equation*}
Then we have
\begin{align*}
\E(g) \copyright \E(f) =\{ ([h_X], [h_Z] ) & \in \E(X) \times \E(Z) \, \, |  \\
& \, \, \text{$\exists [h_Y] \in \E(Y)$ 
such that $([h_X], [h_Y]) \in \E(f)$ and $([h_Y], [h_Z) \in \E(g)$} \}.
\end{align*}
Therefore for any element $([h_X], [h_Z]) \in \E(g) \copyright \E(f)$, we have the following homotopy commutative diagram:
\begin{equation*}
\xymatrix
{ X \ar[r]^{h_X} \ar[d]_f & X \ar[d]^f \\
Y \ar[r]_{h_Y} \ar[d]_g & Y \ar[d]^g\\
Z \ar[r]_{h_Z} & Z,
}
\end{equation*}
which implies the following homotopy commutative diagram:
\begin{equation*}\label{e(f)-gh}
\xymatrix
{ X \ar[r]^{h_X} \ar[d]_{g \circ f}  & X \ar[d]^{g \circ f} \\
Z \ar[r]_{h_Z} & Z.
}
\end{equation*}
Namely we have $([h_X], [h_Z]) \in \E(g \circ f)$. Therefore we get $\E(g) \copyright \E(f)  \subset \E(g \circ f)$, i.e, 
$$\E(g) \copyright \E(f)  \Longrightarrow \E(g \circ f).$$
\end{proof}
\begin{prop} Let $\mathscr Top$ be the category of topological spaces and continuous maps and $\op{Corr}_{\mathscr Top}$ be a 2-strict category of correspondences of topological spaces (called topological correspondences), i.e., a $1$-morphism from $X$ to $Y$ is a correspondence $X \xleftarrow {p_1} S \xrightarrow {p_2} Y$ such that $S$ is a subspace of $X \times Y$ and $p_1, p_2$ are projections restricted to the subspace $S$. A $2$-morphism $\alpha: S_1 \to S_2$ for two correspondences $X \xleftarrow {p_1} S_1 \xrightarrow {p_2} Y$ and $X \xleftarrow {p_1} S_2 \xrightarrow {p_2} Y$ is defined by the inclusion $S_1 \hookrightarrow S_2$. The composition of two topological correspondences is defined in the same way as in the above strict $2$-category of correspondences of groups. Let us define $\Gamma: \mathscr Top \to \op{Corr}_{\mathscr Top}$ as follows:
\begin{enumerate}
\item For a topological space $X$, $\Gamma (X) =X$,
\item For a continuous map $f: X \to Y$, we define $\Gamma(f)$ as the graph correspondence of $f$, i.e.,
$$X \xleftarrow {p_1} \Gamma_f \xrightarrow {p_2} Y$$
where $\Gamma_f:=\{(x, f(x)) \in X \times Y \, \, | \, \, x \in X \}.$
\end{enumerate}
Then $\Gamma: \mathscr Top \to \op{Corr}_{\mathscr Top}$ is a strict functor, i.e., for two continuous maps $f:X \to Y$ and $g:Y \to Z$ we have
$$\Gamma (g) \copyright \Gamma(f) = \Gamma (g \circ f).$$
\end{prop}
\begin{proof} It suffices to show that $\Gamma_g \copyright \Gamma_f = \Gamma_{g \circ f}$.
\begin{align*}
(x, z)  \in \Gamma_g \copyright \Gamma_f & \Longleftarrow \text{$\exists y \in Y$ such that $(x,y) \in \Gamma_f$ and $(y,z) \in \Gamma_g$} \\
& \Longleftrightarrow \text{$y=f(x)$ and $z = g(y)$ for $x \in X$} \\
& \Longleftrightarrow \text{$z = g(f(x)) =(g \circ f)(x)$ for $x \in X$} \\
& \Longleftrightarrow (x, z) \in \Gamma_{g \circ f} \\
\end{align*}
\end{proof}
\begin{thm}\label{top-corr} The above Lax functor $\E: \mathscr Top \to \op{Corr}_{\mathscr Gp}$ is factored through the above strict $2$-category $\op{Corr}_{\mathscr Top}$
$$
\xymatrix{
\mathscr Top \ar[dr]_{\Gamma} \ar[rr]^{\E} && \op{Corr}_{\mathscr Gp} \\
& \op{Corr}_{\mathscr Top} \ar[ur]_{\widetilde \E}
}
$$
Namely, we have  $\E= \widetilde \E \circ \Gamma$. 
Here $\widetilde \E:\op{Corr}_{\mathscr Top} \to \op{Corr}_{\mathscr Gp}$ is defined by
\begin{itemize}
\item $\widetilde \E (X) = \E(X)$ for a topological space $X$,
\item $\widetilde \E(X \xleftarrow {p_1} S \xrightarrow {p_2} Y):= \E(X) \leftarrow \E(p_2) \copyright \E(p_1)^t \rightarrow \E(Y)$
for a topological correspondence $X \xleftarrow {p_1} S \xrightarrow {p_2} Y$.
Here we consider the following correspondences of groups:
\begin{equation*}
\xymatrix@!0{
&& \E(p_2) \copyright \E(p_1)^t \ar@{~}[dl] \ar@{~}[dr] &&&  \\
& \E(p_1)^t \ar[dl] \ar[dr] && \E(p_2) \ar[dl] \ar[dr] &\\
\E(X)  && \E(S) && \E(Y)
}
\end{equation*}
Note that for $p_1:S \to X$ we have the correspondence $\E(S) \leftarrow \E(p_1) \rightarrow \E(X)$ and we take its transpose $\E(X) \leftarrow \E(p_1)^t \rightarrow \E(S)$. 
\end{itemize}
\end{thm}
\begin{proof} 
We want to show that $\E(f) = \widetilde \E (\Gamma (f))$, i.e., $\E(f) = \widetilde \E (\Gamma_f)$ for the graph correspondence $X \xleftarrow {p_1} \Gamma_f \xrightarrow {p_2} Y$ of a continuous map $f:X \to Y$. First we note that $p_1: \Gamma_f \xrightarrow {\cong} X$ is a homeomorphism defined by $p_1(x, f(x)) = x$ and its inverse map $p_1^{-1}$ is $\op{id_X} \times f: X \to \Gamma_f$ defined by $(\op{id_X} \times f)(x) :=(x, f(x))$. $p_2: \Gamma_f \to Y$ is defined by $p_2(x, f(x)) = f(x)$. Thus $f \circ p_1 = p_2$, namely the following diagram commutes:
\begin{equation}\label{triangle}
\xymatrix@!0{
& \Gamma_f \ar[dl]_{p_1}^{\cong} \ar[dr]^{p_2} & \\
X \ar[rr]_f && Y 
}
\end{equation}
First we note that since $p_1:\Gamma_f \to X$ is a homeomorphism, for any $[h_X] \in \E(X)$ we have $[h_{\Gamma_f}]:=[p_1^{-1}][h_X] [p_1] \in \E(\Gamma_f)$ and conversely for any $[h_{\Gamma_f}] \in \E(\Gamma_f)$ we have $[h_X]:=[p_1][h_{\Gamma_f}][p_1]^{-1}$.
Thus we have the following homotopy commutative diagram:
$$
\xymatrix{
\Gamma_f \ar[d]_{p_1}^{\cong}  \ar[rr]^{h_{\Gamma_f}} && \Gamma_f \ar[d]^{p_1}_{\cong} \\
X \ar[rr]_{h_X} && X.
}
$$
Therefore we have
$$\E(p_1) = \{ ([h_{\Gamma_f}], [h_X]) \in \E(\Gamma_f) \times \E(X) \, \, | \, \, [p_1][h_{\Gamma_f}]= [h_X][p_1] \}.$$
Hence we have
\begin{align*}
\E(p_1)^t & = \{ ([h_X],[h_{\Gamma_f}] ) \in \E(X) \times \E(\Gamma_f)  \, \, | \, \, [p_1][h_{\Gamma_f}]= [h_X][p_1] \}\\
& = \{ ([h_X], [p_1]^{-1}[h_X][p_1]) \, \, | \, \, [h_X] \in \E(X) \}
\end{align*}
We also note that we have the following isomorphisms
\begin{align*}
\pi_1 &:\E(p_1)^t \xrightarrow {\cong} \E(X), \quad \pi_1 \left ([h_X], [p_1]^{-1}[h_X][p_1] \right ) = [h_X],\\
\pi_2 & :\E(p_1)^t \xrightarrow {\cong}  \E(\Gamma_f), \quad \pi_2 \left ([h_X], [p_1]^{-1}[h_X][p_1] \right ) = [p_1]^{-1}[h_X][p_1],\\
\pi_2 \circ \pi_1^{-1} &: \E(X) \xrightarrow {\cong}  \E(\Gamma_f), \quad \pi_1 \circ \pi_2^{-1}([h_X]):= [p_1]^{-1}[h_X][p_1].
\end{align*}
For the graph correspondence $X \xleftarrow {p_1} \Gamma_f \xrightarrow {p_2} Y$, we have that $\widetilde \E (\Gamma (f)) =  \E(p_2) \copyright \E(p_1)^t$, where we use the following diagram:
\begin{equation*}
\xymatrix@!0{
&& \E(p_2) \copyright \E(p_1)^t \ar@{~}[dl] \ar@{~}[dr] &&&  \\
& \E(p_1)^t \ar[dl] \ar[dr] && \E(p_2) \ar[dl] \ar[dr] &\\
\E(X)  && \E(\Gamma_f) && \E(Y)
}
\end{equation*}
\begin{align*}
& ([h_X], [h_Y]) \in \E(p_2) \copyright \E(p_1)^t \\
& \Longleftrightarrow
 \text{$\exists [h_{\Gamma_f}]=[p_1]^{-1}[h_X][p_1] \in \E(\Gamma_f)$ such that $([h_X], [h_{\Gamma_f}]) \in \E(p_1)^t$ and $([h_{\Gamma_f}], [h_Y]) \in \E(p_2)$.}\\
 \end{align*}
 Here we have the following commutative diagram:
 $$
 \xymatrix{
 X \ar@<0.8ex>[d]^{p_1^{-1}} \ar[rr]^{h_X} && X \ar@<0.8ex>[d]^{p_1^{-1}} \\
 \Gamma_f \ar@<1ex>[u]^{p_1} \ar[d]_{p_2}  \ar[rr]_{h_{\Gamma_f}} && \Gamma_f \ar@<1ex>[u]^{p_1} \ar[d]^{p_2}\\
 Y \ar[rr]_{h_Y} && Y
 }
 $$
 It follows from the commutative triangle (\ref{triangle}) we have that $p_2\circ p_1^{-1}=f:X \to Y$, thus we get the following (homotopy) commutative diagram
 $$
 \xymatrix{
 X \ar[d]_f \ar[rr]^{h_X} && X \ar[d]^f\\
 Y \ar[rr]_{h_Y} && Y
 }
 $$
 Hence we have $([h_X], [h_Y]) \in \E(f)$, thus $\E(p_2) \copyright \E(p_1)^t \subset \E(f)$.
 
 Conversely let $([h_X], [h_Y]) \in \E(f)$. We let $[h_{\Gamma_f}]:=[p_1]^{-1}[h_X][p_1]$ and let us consider the following commutative diagram (note that $f \circ p_1 =p_2$ (\ref{triangle})):
$$\xymatrix{
\Gamma_f  \ar@/_25pt/[dd]_{p_2} \ar[d]_{p_1}^{\cong} \ar[rr]^{h_{\Gamma_f}} && \Gamma_f \ar[d]^{p_1}_{\cong}  \ar@/^25pt/[dd]^{p_2} \\
X \ar[d]_f \ar[rr]_{h_X}  && X \ar[d]^f \\
 Y \ar[rr]_{h_Y} && Y
}
$$
Hence we have that 
$$\text{$\exists [h_{\Gamma_f}] \in \E(\Gamma_f)$ such that  $([h_X], [h_{\Gamma_f}]) \in \E(p_1)^t$ and $([h_{\Gamma_f}],[h_Y]) \in \E(p_2)$}.$$
Therefore we have that $([h_X], [h_Y]) \in \E(p_2) \copyright \E(p_1)^t.$
\end{proof}
\begin{cor} Let $\op{Corr}_{\mathscr Top}^{\Gamma}$ be the subcategory of $\op{Corr}_{\mathscr Top}$ such that its $1$-morphisms are the topological graph correspondences. Then $\widetilde {\E}: \op{Corr}_{\mathscr Top}^{\Gamma} \to \op{Corr}_{\mathscr Gp}$ is a Lax functor, i.e., 
$$\widetilde {\E}(\Gamma_g) \copyright \widetilde {\E}(\Gamma_f) \Longrightarrow \widetilde {\E}(\Gamma_g \copyright \Gamma_f).$$
\end{cor}
\begin{proof}
$$\widetilde {\E}(\Gamma_g) \copyright \widetilde {\E}(\Gamma_f)  = \E(f) \copyright \E(f) \Longrightarrow \E(g \circ f) = \widetilde {\E} (\Gamma_{g \circ f}) = \widetilde {\E}(\Gamma_g \copyright \Gamma_f).$$
\end{proof} 
\begin{rem}
It remains to see whether or not $\widetilde {\E}: \op{Corr}_{\mathscr Top} \to \op{Corr}_{\mathscr Gr}$ is also a Lax funcor, i.e. the above Lax functor 
$\widetilde {\E}: \op{Corr}_{\mathscr Top}^{\Gamma} \to \op{Corr}_{\mathscr Gp}$ can be extended to the whole category $\op{Corr}_{\mathscr Top}$.
\end{rem}
Motivated by the proof of Theorem \ref{top-corr}, we can show the following.
\begin{cor} Let $\mathscr Top^{\mathscr{HE}}$ be the subcategory of $\mathscr Top$ consisting of homotopy equivalences. For a topological space $X$ we define $\E_*(X) := \E(X)$ and for a homotopy equivalence $h:X \to Y$, $\E_*(h):=\op{Ad}_{[h]}: \E(X) \to \E(Y)$ by, for $[h_X] \in \E(X)$
$$\op{Ad}_{[h]}([h_X]):= [h][h_X][h]^{-1},$$
where we consider the following homotopy commutative diagram
\begin{equation}\label{hhhh}
\xymatrix{
X \ar[d]_h \ar[rr]^{h_X} && X \ar[d]^h \\
Y \ar[rr]_{h \circ h_X \circ h^{-1}} && Y.
}
\end{equation}
Then the following diagrams commute:
$$
\xymatrix
{\mathscr Top^{\mathscr{HE}} \ar[d]_{\Gamma} \ar[drr]_{\E} \ar[rr]^{\E_*} && \mathscr Gr \ar[d]^{\Gamma}\\
\op{Corr}_{\mathscr Top} \ar[rr]_{\widetilde \E} && \op{Corr}_{\mathscr Gr}.
}
$$
Here $\Gamma: \mathscr Gr \to \op{Corr}_{\mathscr Gr}$ is the $\mathscr Gr$-version of the topological one $\Gamma: \mathscr Top \to \op{Corr}_{\mathscr Top}.$
\end{cor}
\begin{proof} First we note that it is not necessary to change $\op{Corr}_{\mathscr Top}$ by $\op{Corr}_{\mathscr Top^{\mathscr {HE}}}$. Since the lower left triangle commutes as shown above, it suffices to show the commutativity of the upper right triangle, i.e., $\Gamma \circ \E_* = \E$. Clearly by the definition, for a topological space $X$ we have $\left (\Gamma \circ \E_* \right) (X) =\E(X).$ For a homotopy equivalence $h:X \to Y$ we have that 
$$\E(h):= \{([h_X], [h_Y]) \in \E(X) \times \E(Y) \, \, | \, \, [h][h_X][h]^{-1}=[h_Y] \}$$
for which see (\ref{hhhh}). Thus, since $[h_Y]= \E_*(h)([h_X])= \op{Ad}_{[h]}([h_X])=[h][h_X][h]^{-1}$, we have
$$\E(h) = \Gamma_{\E_*(h)} = \{([h_X], \E_*(h)([h_X])) \, \, | \, \, [h_X] \in \E(X)\}.$$
Therefore we have that $(\Gamma \circ \E_*)(h) = \E(h)$.
\end{proof}
\section{On $\E$ of fibrations}
In this section we consider maps $f:X \to Y$ and $g:Y \to Z$ for which the Lax functor $\E$ becomes ``strict", i.e., it satisfies 
$$\E(g) \copyright \E(f) = \E(g \circ f).$$
So, we give the following name to such a pair:
\begin{defi} 
If  $\E(g) \copyright \E(f) = \E(g \circ f)$ holds, then $(f,g)$ is called \emph{an $\E$-strict-functor pair}. 
\end{defi}
\begin{prop}\label{str}
If $\xi :X \overset{f} \hookrightarrow Y \xrightarrow {g} Z$ is a trivial fibration, i.e., $Y\simeq X\times Z$, and $Z$ is path connected,  then $(f,g)$ is an $\E$-strict-functor pair.
\end{prop}

\begin{proof} 
We can take the fibration $\xi$ as $X \overset{i_1} \hookrightarrow  X\times Z \xrightarrow {p_2} Z$  by homotopy.
Here the inclusion $i_1:X \to X \times Z$ is defined by $i_1(x)=(x, z_0)$ for some (in fact, any) point $z_0 \in Z$ and $p_2:X \times Z \to Z$ is the projection to the second factor $Z$. We want to show that
$$\E( p_2) \copyright \E(i_1) =\E(p_2 \circ i_1).$$
Since the composition $p_2 \circ i_1:X \to Z$ is a constant map and $Z$ is path connected, it follows from Lemma \ref{ccc} that $\E(p_2 \circ i_1)=\E(X) \times \E(Z)$. Conversely we consider the following arbitrary element
$$([h_X], [h_Z]) \in \E(X) \times \E(Z) \, (= \E(p_2 \circ i_1)).$$
Then, using $h_X, h_Z$, we have the following \emph{homotopy}\footnote{Note that the second diagram is commutative, but the first diagram is not commutative, because we have that 
$$(i_1 \circ h_X)(x) = (h_X(x), z_0) \not = (h_X(x), h_Z(z_0) ) = ((h_X \times h_Z) \circ i_1)(x).$$ However the first diagram  is \emph{homotopy} commutative diagram, because $Z$ is path connected.} commutative diagrams
\begin{equation*}
\xymatrix
{ X \ar[rr]^{h_X} \ar[d]_{i_1} && X \ar[d]^{i_1}\\
X\times Z \ar[rr]_{h_X \times h_Z} && X\times Z,
}
\qquad 
\xymatrix
{ X\times Z \ar[rr]^{h_X \times h_Z} \ar[d]_{p_2} && X\times Z \ar[d]^{p_2}\\
Z \ar[rr]_{h_Z} &&  Z
}
\end{equation*}
Therefore we have  
$$([h_X], [h_X \times h_Z]) \in \E(i_1), \quad ([h_X \times h_Z], [h_Z]) \in \E(p_2),$$
which implies that $([h_X], [h_Z]) \in \E(p_2) \copyright \E(i_1).$ Therefore we have $ \E(p_2) \copyright \E(i_1) = \E(p_2 \circ i_1).$
\end{proof}
A famous theorem of John Frank Adams \cite{Adams} on Hopf invariant one is that \emph{a map $g:S^{2n-1} \to S^n$ of Hopf invariant one exists only for $n=2,4$ and $8$}, i.e., the following well-known Hopf fibrations (e.g., see \cite[Example 4.45, Example 4.46, Example 4.47]{Hatcher}): 
\begin{enumerate}
\item $S^1\overset{f}{\hookrightarrow } S^3\overset{g}{\to} S^2$,
\item $S^3\overset{f}{\hookrightarrow } S^7\overset{g}{\to} S^4$,
\item $S^7\overset{f}{\hookrightarrow } S^{15}\overset{g}{\to} S^8$.
\end{enumerate}
As to these three Hopf fibrations we can show that the first Hopf fibration $(f,g)$ \emph{is} an $\E$-strict-functor pair, but the other two Hopf fibrations $(f,g)$ are \emph{not} $\E$-strict-functor pairs. To show this, we use $\pi_i(S^n)=0$ for $i <n$ (see \cite[Corollary 4.9]{Hatcher}, which follows from the \emph{Cellular Approximation Theorem} \cite[Theorem 4.8]{Hatcher}), thus any continuous map $S^i \to S^n$ (base point preserving) is homotopic to a constant map $c:S^i \to S^n$ (i.e. $c(S^i)$ is the base point).
\begin{thm} 
\begin{enumerate}
\item For the Hopf fibration $S^1\overset{f}{\hookrightarrow } S^3\overset{g}{\to} S^2$, $(f,g)$ is an $\E$-strict-functor pair. 
\item For the Hopf fibrations $S^3\overset{f}{\hookrightarrow } S^7\overset{g}{\to} S^4$ and $S^7\overset{f}{\hookrightarrow } S^{15}\overset{g}{\to} S^8$, $(f,g)$ are \emph{not} $\E$-strict-functor pairs.
\end{enumerate}
\end{thm} 
\begin{proof} 
(1) $S^1\overset{f}{\hookrightarrow } S^3\overset{g}{\to} S^2$: It follows from  \cite[Example 4.2(i)]{KO} that the following diagram 
\begin{equation}\label{s3-s2}
\xymatrix
{ S^3 \ar[rr]^{k^2\cdot\op{id}_{S^3}} \ar[d]_g && S^3 \ar[d]^g\\
S^2 \ar[rr]_{k \cdot \op{id}_{S^2}} && S^2
}
\end{equation}
is homotopy commutative for any $k\in \Z$. 
For the sake of convenience, the homotopy class $[\op{id}_X]$ of the identity map $\op{id}_X$ shall be simply denoted by $1_X$. It is well-known that $\E(S^n) = \{1_{S^n}, -1_{S^n}\}$ for any $n$, so we have $\E(S^2)=\{1_{S^2}, -1_{S^2}\}$ and $\E(S^3)=\{1_{S^3}, -1_{S^3}\}$. Hence in our case $k$ can be only $1$ and $-1$ in the above diagram (\ref{s3-s2}), thus 
$$\E(g)= \{(1_{S^3}, 1_{S^2}), (1_{S^3}, -1_{S^2})\} \, \, (\subset \E(S^3) \times \E(S^2)).$$
On the other hand, since $f: S^1 \hookrightarrow S^3$ is homotopic to a constant map as remarked above, it follows from Lemma \ref{ccc-1} that we have
$$ \E(f) = \E(S^1) \times \E(S^3) =\{(1_{S^1}, 1_{S^3}), (1_{S^1}, -1_{S^3}), (-1_{S^1}, 1_{S^3}), (-1_{S^1}, -1_{S^3})\}$$
Thus we have 
\begin{align*}
\E(g)\copyright \E(f) & =\{(1_{S^1}, 1_{S^2}), (1_{S^1}, -1_{S^2}), (-1_{S^1}, 1_{S^2}), (-1_{S^1}, -1_{S^2})\} \\
& =\E(S^1)\times \E(S^2) \\
& =\E(g\circ f) \quad \text{(since $g \circ f:S^1 \to S^2$ is a constant map and use Lemma \ref{ccc})}
\end{align*}
Thus $(f,g)$ is an $\E$-strict-functor pair\footnote{Note that in the above computation of $\E(g)\copyright \E(f)$ we use only two elements $(1_{S^1}, \underline{1_{S^3}})$ and $(-1_{S^1}, \underline{1_{S^3}})$ in the above group $\E(f)$, because we have that $\E(g)= \{(\underline{1_{S^3}}, 1_{S^2}), (\underline{1_{S^3}}, -1_{S^2})\}$. (Look at the underlined $\underline{\text{elements}}$.)}. 

(2) $S^3\overset{f}{\hookrightarrow } S^7\overset{g}{\to} S^4$: It follows from  \cite[Example 4.2(ii)]{KO} that the following diagram 
\begin{equation}\label{s7-s4}
\xymatrix
{ S^7 \ar[rr]^{k^2 \cdot \op{id}_{S^7}} \ar[d]_g && S^7 \ar[d]^g\\
S^4 \ar[rr]_{k \cdot \op{id}_{S^4}} && S^4
}
\end{equation}
is homotopy commutative for any $k\in \Z$ \emph{such that $k(k-1) \equiv 0 \, (\op{mod} 8)$}. Hence, in our case it is only possible that $k = 1$ in the above diagram (\ref{s7-s4}). Thus we have
\begin{equation}\label{7-4}
\E(g) = \{(1_{S^7}, 1_{S^4})\} \, (\subset \E(S^7) \times \E(S^4)).
\end{equation}
Since $f: S^3 \to S^7$ is homotopic to a constant map, we have
$$ \E(f) = \E(S^3) \times \E(S^7) =\{(1_{S^3}, 1_{S^7}), (1_{S^3}, -1_{S^7}), (-1_{S^3}, 1_{S^7}), (-1_{S^3}, -1_{S^7})\}.$$
Hence, because of (\ref{7-4}) we get 
$$\E(g) \copyright \E(f) =\{(1_{S^3}, 1_{S^4}), (-1_{S^3}, 1_{S^4}) \},$$
which is \emph{not} equal to $\E( g \circ f) = \E(S^3) \times \E(S^4).$ Therefore this $(f,g)$ is not an $\E$-strict-functor pair.

(3) $S^7\overset{f}{\hookrightarrow } S^{15}\overset{g}{\to} S^8$: It follows from  \cite[Example 4.2(iii)]{KO} that the following diagram 
\begin{equation}\label{s15-s8}
\xymatrix
{ S^{15} \ar[rr]^{k^2 \cdot \op{id}_{S^{15}}} \ar[d]_g && S^{15} \ar[d]^g\\
S^8 \ar[rr]_{k \cdot \op{id}_{S^8}} && S^8
}
\end{equation}
is homotopy commutative for any $k\in \Z$ \emph{such that $k(k-1) \equiv 0 \, (\op{mod} 16)$}. Thus, in our case, as in the above (2), it is only possible that $k=1$ in the above diagram (\ref{s15-s8}). Thus we have
$$\E(g) = \{(1_{S^{15}}, 1_{S^8})\} (\subset \E(S^{15}) \times \E(S^8)).$$
Since $f: S^7 \to S^{15}$ is homotopic to a constant map, we have
$$ \E(f) = \E(S^7) \times \E(S^{15}) =\{(1_{S^7}, 1_{S^{15}}), (1_{S^7}, -1_{S^{15}}), (-1_{S^7}, 1_{S^{15}}), (-1_{S^7}, -1_{S^{15}})\}.$$
Hence we get
$$\E(g) \copyright \E(f) =\{(1_{S^7}, 1_{S^8}), (-1_{S^7}, 1_{S^8}) \},$$
which is \emph{not} equal to $\E( g \circ f) = \E(S^7) \times \E(S^8).$ Therefore this $(f,g)$ is not an $\E$-strict-functor pair.
\end{proof}
\begin{thm}\label{contra} Let $f:X \to Y$ and $g:Y \to Z$ be continuous maps such that $Y$ is a contractible space and $Z$ is path connected. Then $(f,g)$ is an $\E$-strict-functor pair.
\end{thm}
\begin{proof} Since $Y$ is contractible, $Y$ is a homotopy equivalence to a point, thus $\E(Y)$ is a trivial group $\{1_Y\}$.
Since $Y$ is contractible, there exists a continuous map 
$$H: Y \times [0, 1] \to Y$$
such that $H(y,0) = y$, i.e., $H|_{Y \times \{0\}}:Y \to Y$ is the identity map and $H(y,1) = y_0$ for some point $y_0 \in Y$, i.e., $H|_{Y \times \{1\}}:Y \to Y$ is a constant map. Namely, the identity map $\op{id}_Y:Y \to Y$ is homotopic to the constant map $c:Y \to Y$ such that $c(Y)=y_0$. Then we consider the following composition of continuous maps $H$ and $f \times \op{id}_{[0,1]}$:
$$H \circ (f \times \op{id}_{[0,1]}) :X \times [0,1] \xrightarrow {f \times \op{id}_{[0,1]}} Y \times [0,1]  \xrightarrow H Y.$$
Which implies that $H \circ (f \times \op{id}_{[0,1]})(y, 0) = f(y)$, i.e., $H \circ (f \times \op{id}_{[0,1]})|_{X \times \{0\}}:X \to Y$ is equal to $f:X \to Y$ and  $H \circ (f \times \op{id}_{[0,1]})(y, 1) = y_0$, i.e., $H \circ (f \times \op{id}_{[0,1]})|_{X \times \{1\}}:X \to Y$ is equal to the constant map $d:X \to Y$ such that $d(X) =y_0$. Hence $f:X \to Y$ is homotopic to a constant map $d:X \to Y$. Since $Y$ is contractible, it is automatically path connected. Thus it follows from Lemma \ref{ccc-1} that 
\begin{equation}\label{e(f)}
\E(f) = \E(X) \times \E(Y) =\E(X) \times \{1_Y\}.
\end{equation}

Now we consider the following composition of continuous maps $H$ and $g:Y \to Z$:
$$g \circ H: Y \times [0, 1] \xrightarrow {H}  Y \xrightarrow g Z.$$
Which implies that $g=(g \circ H)|_{Y \times \{0\}}:Y \to Z$ is homotopic to the constant map $e=(g \circ H)|_{Y \times \{1\}}:Y \to Z$ such that $e(Y) =g(y_0)$. Thus, since $Z$ is path connected, it follows from Lemma \ref{ccc-1} that 
\begin{equation}\label{e(g)}
\E(g) = \E(Y) \times \E(Z) =\{1_Y \} \times \E(Z).
\end{equation}
(\ref{e(f)}) and (\ref{e(g)}) imply that 
$$\E(g) \copyright \E(f) = \E(X) \times \E(Z)$$
which implies that $\E(g) \copyright \E(f) = \E( g \circ f)$, thus $(f,g)$ is an an $\E$-strict-functor pair. It is becasue $\E(g) \copyright \E(f) \Longrightarrow \E(g \circ f)$, i.e., $\E(g) \copyright \E(f) \subset \E(g \circ f)$ always holds and by the definition $\E(g \circ f) \subset \E(X) \times \E(Z)$. Another reason is that by the above arguments we see that $g \circ f:X \to Z$ is homotopic to the constant map $e \circ d:X \to Z$ such that $(e \circ d)(X) = g(y_0)$, thus it follows from Lemma \ref{ccc-1} that $\E(g \circ f)= \E(X) \times \E(Z)$.
\end{proof}
In the above theorem contractibility of $Y$ is a key. As a corollary of this theorem we can get the following
\begin{cor} 
\begin{enumerate}
\item Let $X$ be path connected and let $\Omega X \overset {f} {\hookrightarrow} PX \xrightarrow {g} X$ be the loop-path fibration. Then $(f,g)$ is an $\E$-strict-functor pair.
\item Let $G \overset {f} {\hookrightarrow} EG \xrightarrow {g} BG$ be the universal principal G bundle for a topological group $G$. Here $BG$ is the classifying space of $G$. Then $(f,g)$ is an an $\E$-strict-functor pair.
\end{enumerate}
\end{cor}
\begin{rem} Note that the classifying space $BG \cong EG/G$ is path connected since $EG$ is contractible, thus path connected, because the quotient space of a path connected space is automatically path connected.
\end{rem}
\begin{rem} 
The case treated in the above Theorem \ref{contra} is the case when the ``fiber-duct" is the same as the ``fiber product". 
$$
\xymatrix@!0{
&& \E(g) \copyright \E(f) =\E(X) \times \E(Z) \ar@{~}[dl] \ar@{~}[dr]  &&& \\
& \E(f) \ar[dl] \ar[dr] && \E(g)  \ar[dl] \ar[dr]  &\\
\E(X)  && \E(Y) && \E(Z)
}
\qquad 
\xymatrix@!0{
&& \E(f) \times_{\E(Y)} \E(g) = \E(X) \times \E(Z) \ar[dl] \ar[dr]  &&& \\
& \E(f) \ar[dl] \ar[dr] && \E(g)  \ar[dl] \ar[dr]  &\\
\E(X)  && \E(Y) && \E(Z)
}
$$
Here is another simple example of the case when ``fiber-duct" is the same as the ``fiber product". Consider the diagonal correspondence of a group $G$, i.e., $G \leftarrow \Delta (G) = \{(g, g) \,| \, g \in G\}\rightarrow G$. Then the ``fiber-duct" $\Delta (G) \copyright \Delta (G) = \Delta (G) \cong G$ and the fiber product $\Delta (G) \times_G \Delta(G) \cong G$ are the same:
$$
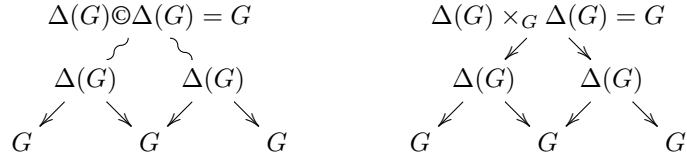
\xymatrix@!0{
&& \Delta (G) \copyright \Delta (G)=G  \ar@{~}[dl] \ar@{~}[dr]  &&& \\
& \Delta (G)\ar[dl] \ar[dr] && \Delta (G) \ar[dl] \ar[dr]  &\\
G  && G && G
}
\qquad 
\xymatrix@!0{
&& \Delta (G) \times_G \Delta (G) = G  \ar[dl] \ar[dr]  &&& \\
& \Delta (G)  \ar[dl] \ar[dr] && \Delta (G)  \ar[dl] \ar[dr]  &\\
G && G && G
}
$$
\end{rem}

\vspace{1cm}

\noindent
{\bf Acknowledgements}

The second author is supported by JSPS KAKENHI Grant Number JP23K03117. \\

\end{document}